\documentclass[12pt,a4paper]{amsart}
\calclayout
\usepackage[utf8]{inputenc}
\usepackage[T1]{fontenc}
\usepackage{mathtools}
\mathtoolsset{showonlyrefs}
\usepackage{mathrsfs}
\usepackage{hyperref}
\usepackage{amsthm}
\usepackage{comment}
\theoremstyle{plain}
\newtheorem{thm}{Theorem}[section]
\newtheorem{lem}[thm]{Lemma}
\newtheorem{prop}[thm]{Proposition}
\newtheorem{cor}[thm]{Corollary}

\theoremstyle{definition}

\newtheorem{asp}{Assumption}
\numberwithin{equation}{section}
\allowdisplaybreaks
\begin{document}
\title
[stable CLT for super-OU processes, II]
{Stable Central Limit Theorems for Super Ornstein-Uhlenbeck Processes, II}
\author
[Y.-X. Ren, R. Song, Z. Sun and J. Zhao]
{Yan-Xia Ren, Renming Song, Zhenyao Sun and Jianjie Zhao}
\address{
	Yan-Xia Ren \\
	LMAM School of Mathematical Sciences \& Center for Statistical Science \\
  	Peking University \\
  	Beijing 100871, P. R. China}
\email{yxren@math.pku.edu.cn}
\thanks{Yan-Xia Ren is supported in part by NSFC (Grant Nos. 11671017  and 11731009) and LMEQF}
\address{
  	Renming Song \\
  	Department of Mathematics \\
  	University of Illinois at Urbana-Champaign \\
  	Urbana, IL 61801, USA}
\email{rsong@illinois.edu}
\thanks{Renming Song is support in part by a grant from the Simons Foundation (\#429343, Renming Song)}
\address{
  	Zhenyao Sun \\
  	Faculty of Industrial Engineering and Management\\
  	Technion, Israel Institute of Technology \\
  	Haifa 3200003, Israel}
\email{zhenyao.sun@gmail.com}
\thanks{Zhenyao Sun is the corresponding author}
\address{
  	Jianjie Zhao \\
  	School of Mathematical Sciences \\
  	Peking University \\
  	Beijing 100871, P. R. China}
\email{zhaojianjie@pku.edu.cn}

\begin{abstract}
	This paper is a continuation of our recent paper (Elect. J. Probab. \textbf{24} (2019), no. 141) and is devoted to the asymptotic behavior of a class of supercritical super Ornstein-Uhlenbeck processes $(X_t)_{t\geq 0}$ with branching mechanisms of infinite second moment. In the aforementioned paper, we proved stable central limit theorems for  $X_t(f) $ for {\it some} functions $f$ of polynomial growth in three different regimes. However, we were not able to prove central limit theorems for $X_t(f) $ for {\it all} functions $f$ of polynomial growth.
	In this note, we show that the limiting stable random variables in the three different regimes are independent, and as a consequence, we get stable central limit theorems for  $X_t(f) $ for {\it all} functions $f$ of polynomial growth.
\end{abstract}
\maketitle

\section{Introduction and main result}
\label{subsec:M}
	Let $d \in \mathbb N:= \{1,2,\dots\}$ and $\mathbb R_+:= [0,\infty)$.
	Let $\xi=\{(\xi_t)_{t\geq 0}; (\Pi_x)_{x\in \mathbb R^d}\}$ be an $\mathbb R^d$-valued Ornstein-Uhlenbeck process (OU process) with generator
\begin{align}
  	Lf(x)
  	= \frac{1}{2}\sigma^2\Delta f(x)-b x \cdot \nabla f(x),
  	\quad  x\in \mathbb R^d, f \in C^2(\mathbb R^d),
\end{align}
	where $\sigma > 0$ and $b > 0$ are constants.
	Let $\psi$ be a function on $\mathbb R_+$ of the form
\begin{align}
  	\psi(z)
  	=- \alpha z + \rho z^2 + \int_{(0,\infty)} (e^{-zy} - 1 + zy) \pi(\mathrm dy),
  	\quad  z \in \mathbb R_+,
\end{align}
	where $\alpha > 0 $, $\rho \geq0$ and $\pi$ is a measure on $(0,\infty)$ with $\int_{(0,\infty)}(y\wedge y^2)\pi(\mathrm dy)< \infty$.
	$\psi$ is referred to as a branching mechanism and $\pi$ is referred to as the L\'evy measure of $\psi$.
	Denote by $\mathcal M(\mathbb R^d)$ ($\mathcal M_c(\mathbb R^d)$) the space of all finite Borel measures (of compact support) on $\mathbb R^d$.
	Denote by $\mathcal B(\mathbb R^d, \mathbb R)$ ($\mathcal B(\mathbb R^d, \mathbb R_+)$) the space of all $\mathbb R$-valued ($\mathbb R_+$-valued)
	Borel functions on $\mathbb R^d$.
	For $f,g\in \mathcal B(\mathbb R^d, \mathbb R)$ and $\mu \in \mathcal M(\mathbb R^d)$, write $\mu(f)= \int f(x)\mu(\mathrm dx)$ and $\langle f, g\rangle = \int f(x)g(x) \mathrm dx$ whenever the integrals make sense.
	We say a real-valued Borel function $f$ on $\mathbb R_+\times \mathbb R^d$ is \emph{locally bounded} if, for each $t\in \mathbb R_+$, we have $ \sup_{s\in [0,t],x\in \mathbb R^d} |f(s,x)|<\infty. $
	For any $\mu \in \mathcal M(\mathbb R^d)$, we write $\|\mu\| = \mu(1)$.
	For any $\sigma$-finite signed measure $\mu$, denote by $|\mu|$ the total variation measure of $\mu$.

	We say that an $\mathcal M(\mathbb R^d)$-valued Hunt process $X = \{(X_t)_{t\geq 0}; (\mathbb{P}_{\mu})_{\mu \in \mathcal M(\mathbb R^d)}\}$ is a \emph{super Ornstein-Uhlenbeck process (super-OU process)} with branching mechanism $\psi$, or a \emph{$(\xi, \psi)$-superprocess}, if for each non-negative bounded Borel function $f$ on $\mathbb R^d$, we have
\begin{align}
  	\label{eq: def of V_t}
  	\mathbb{P}_{\mu}[e^{-X_t(f)}]
  	= e^{-\mu(V_tf)},
   	\quad t\geq 0, \mu \in \mathcal M(\mathbb R^d),
\end{align}
	where $(t,x) \mapsto V_tf(x)$ is the unique locally bounded non-negative solution to the equation
\begin{align}
  	V_tf(x) + \Pi_x \Big[ \int_0^t\psi (V_{t-s}f(\xi_s) ) \mathrm ds\Big]
	= \Pi_x [f(\xi_t)],
	\quad x\in \mathbb R^d, t\geq 0.
\end{align}	
	The existence of such super-OU process $X$ is well known, see \cite{Dynkin1993Superprocesses,Li2011Measure-valued} for instance.

	There have been many central limit theorem type results for branching processes, branching diffusions and superprocesses, under the second moment condition.
	See \cite{AdamczakMilos2015CLT, AsmussenHering1983Branching, Athreya1969Limit,Athreya1969LimitB,Athreya1971Some, Heyde1970A-rate, HeydeBrown1871An-invariance, HeydeLeslie1971Improved, KestenStigum1966Additional,KestenStigum1966A-limit,Milos2012Spatial,RenSongZhang2014Central,RenSongZhang2014CentralB, RenSongZhang2015Central,RenSongZhang2017Central}.
	For a detailed literature review, see \cite[Section 1.1]{RenSongSunZhao2019Stable}.
	There are also  central  limit theorem type results for supercritical branching processes and branching Markov processes with branching mechanisms of infinite second moment.
	For earlier papers, see \cite{Asmussen76Convergence, Heyde1971Some}. 	
	Recently, Marks and Milo\'s \cite{MarksMilos2018CLT} established some spatial central limit theorems for supercritical branching OU processes with a special stable offspring distribution.
	In \cite{RenSongSunZhao2019Stable}, we established stable central limit theorems for super-OU processes $X$ with branching mechanisms $\psi$ satisfying the following two assumptions.

\begin{asp}[Grey's condition]
\label{asp: Greys condition}
	There exists $z' > 0$ such that $\psi(z) > 0$ for all $z>z'$ and  $\int_{z'}^\infty \psi(z)^{-1} \mathrm dz < \infty$.
\end{asp}

\begin{asp}
\label{asp: branching mechanism}
  	There exist constants $\eta > 0$ and $\beta \in (0,1)$ such that
\begin{align}
    \int_{(1,\infty)}y^{1+\beta +\delta}\Big|\pi(\mathrm dy)-\frac{\eta \mathrm dy}{\Gamma(-1-\beta)y^{2+\beta}}\Big| <\infty
\end{align}
	for some $\delta > 0$.
\end{asp}

	It is known (see \cite[Theorems 12.5 \& 12.7]{Kyprianou2014Fluctuations} for example) that, under Assumption \ref{asp: Greys condition}, the \emph{extinction event}
$
	D
	:=\{\exists t\geq 0~\text{such that}~ \|X_t\| =0 \}
$ 
	is non-trivial with respect to $\mathbb P_\mu$ for each  $\mu \in \mathcal M(\mathbb R^d)\setminus\{0\}$.
	It follows from \cite[Lemma 2.2]{RenSongSunZhao2019Stable}  that, if Assumption \ref{asp: branching mechanism} holds, then $\eta$ and $\beta$ are uniquely determined by the L\'evy measure $\pi$.

	We now recall some notation and basic facts from \cite{RenSongSunZhao2019Stable}.
	We use  $(P_t)_{t\geq 0}$ to denote the transition semigroup of $\xi$.	
	Define $P^{\alpha}_t f(x) := e^{\alpha t} P_t f(x) = \Pi_x [e^{\alpha t}f(\xi_t)]$ for each $x\in \mathbb R^d$, $t\geq 0$ and $f\in \mathcal B(\mathbb R^d, \mathbb R_+)$.
	It is known that $\mathbb{P}_{\mu}[X_t (f)]  = \mu( P^\alpha_t f)$ for all $\mu\in \mathcal M(\mathbb R^d)$, $t\geq 0$ and $f\in \mathcal B(\mathbb R^d, \mathbb R_+)$.
	The OU process $\xi$ has an invariant probability on $\mathbb R^d$:
\begin{align}
  	\varphi(x)\mathrm dx
  	:=\Big (\frac{b}{\pi \sigma^2}\Big )^{d/2}\exp \Big(-\frac{b}{\sigma^2}|x|^2 \Big)\mathrm dx.
\end{align}
	Let $L^2(\varphi)$ be the Hilbert space with inner product
\begin{align}
  	\langle f_1, f_2 \rangle_{\varphi}
  	:= \int_{\mathbb R^d}f_1(x)f_2(x)\varphi(x) \mathrm dx, \quad f_1,f_2 \in L^2(\varphi).
\end{align}
	Let $\mathbb Z_+ := \mathbb N\cup\{0\}$.
	It is known that $(P_t)_{t\geq 0}$ is a strongly continuous semigroup in $L^2(\varphi)$ and its generator $L$ has discrete spectrum $\sigma(L)= \{-bk: k \in \mathbb Z_+\}$.	
		The eigenfunctions of $L$ consists a family of polynomials $\{\phi_p:p\in \mathbb Z_+^d\}$ which forms a complete orthonormal basis of $L^2(\varphi)$.
For each $p\in \mathbb Z_+^d$, 
$\phi_p$ is an eigenfunction of $L$ corresponding to the eigenvalue $b|p|$, where $|p|:= \sum_{k=1}^d p_k$.
	For each function $f\in L^2(\varphi)$, define the order of $f$ as
$\kappa_f
  	:= \inf \left \{k\geq 0: \exists ~ p\in \mathbb Z_+^d , {\rm ~s.t.~} |p|=k {\rm ~and~}  \langle f, \phi_p \rangle_{\varphi}\neq 0\right \}$
with the convention that $\inf\emptyset=\infty.$
	
	For $p\in \mathbb{Z}_+^d$, define
$H_t^p
    	:= e^{-(\alpha-|p|b)t}X_t(\phi_p), t\geq 0.$
	For each $u \neq -1$, we write $\tilde u = u/(1+ u)$.
	We have shown in \cite[Lemma 3.2]{RenSongSunZhao2019Stable} that
	for any
	$\mu\in \mathcal M_c(\mathbb R^d)$, $(H_t^p)_{t\geq 0}$ is a $\mathbb{P}_{\mu}$-martingale.
	Furthermore, if $\alpha \tilde \beta>|p|b$, then for every $\gamma\in (0, \beta)$ and $\mu\in \mathcal M_c(\mathbb R^d)$,  $(H_t^p)_{t\geq 0}$ is a $\mathbb{P}_{\mu}$-martingale bounded in $L^{1+\gamma}(\mathbb{P}_{\mu})$;
	thus $H^p_{\infty}:=\lim_{t\rightarrow \infty}H_t^p$ exists $\mathbb{P}_{\mu}$-almost surely and in $L^{1+\gamma}(\mathbb P_\mu)$.
We will write $H^0_t$ and $H^0_\infty$ as $H_t$ and $H_\infty$, respectively.

	Denote by $\mathcal P \subset L^2(\varphi)$ the class of functions of polynomial growth on $\mathbb R^d$, i.e.,
$\mathcal{P}
  	:= \{f\in \mathcal B(\mathbb R^d, \mathbb R):\exists C>0, n \in \mathbb Z_+ \text{~s.t.~} \forall x\in \mathbb R^d, |f(x)|\leq C(1+|x|)^n \}.$
	Define
	$\mathcal C_\mathrm s := \mathcal P \cap \overline{\operatorname{Span}} \{ \phi_p: \alpha \tilde \beta < |p| b \},
		~
	\mathcal C_\mathrm c   := \mathcal P \cap \operatorname{Span} \{ \phi_p : \alpha \tilde \beta = |p| b \} $, and
	$ \mathcal C_\mathrm l   := \mathcal P \cap \operatorname{Span} \{ \phi_p: \alpha \tilde \beta > |p| b \}$.
	Note that $\mathcal C_\mathrm s$ is an infinite dimensional space, $ \mathcal C_\mathrm l$ and $\mathcal C_\mathrm c$ are finite dimensional spaces, and $\mathcal C_c$ might be empty.
	Define a semigroup
\begin{align}\label{e:Tt}
	T_t f
	:= \sum_{p \in \mathbb Z_+^d} e^{-\big| |p|b - \alpha \tilde \beta \big|t} \langle f, \phi_p \rangle_{\varphi} \phi_p,
	\quad t\geq 0, f\in \mathcal P,
\end{align}
	and a family of functionals
\begin{align}\label{eq:M.13}
	m_t[f]
	:= \eta \int_0^t \mathrm du \int_{\mathbb R^d} \big(-iT_u f(x)\big)^{1+\beta} \varphi(x) \mathrm dx,
	\quad 0 \leq t< \infty, f\in \mathcal P.
\end{align}
	We have shown in \cite[Lemma 2.6 and Proposition 2.7]{RenSongSunZhao2019Stable} that,
\begin{equation} \label{eq:I:R:3}
\begin{minipage}{0.9\textwidth}
	for each $f\in \mathcal P$, there exists a $(1+\beta)$-stable random variable $\zeta^f$ with characteristic function $\theta \mapsto e^{m[\theta f]}, \theta \in \mathbb R,$ where
\begin{equation}
	m[f]
	:=
\begin{cases}
	\lim_{t\to \infty} m_t[f], &
	f \in \mathcal C_\mathrm s \oplus \mathcal C_\mathrm l, \\
	\lim_{t\to \infty} \frac{1}{t} m_t[f], & f\in \mathcal P \setminus \mathcal C_\mathrm s \oplus \mathcal C_\mathrm l.
\end{cases}
\end{equation}
\end{minipage}
\end{equation}
		For each $\mu \in \mathcal M(\mathbb R^d)\setminus\{0\}$, write $\widetilde {\mathbb P}_\mu(\cdot):= \mathbb P_\mu(\cdot | D^c).$
	We also proved in \cite[Theorem 1.6]{RenSongSunZhao2019Stable} that
	\begin{equation}\label{eq:oldResult}
\begin{minipage}{0.9\textwidth}
	if $\mu\in \mathcal M_\mathrm c(\mathbb R^d)\setminus \{0\}$, $f_\mathrm s\in \mathcal C_\mathrm s$, $f_\mathrm c \in \mathcal C_\mathrm c$ and $f_\mathrm l \in \mathcal C_\mathrm l$, then under $\mathbb {\widetilde P}_\mu$,
\[
\begin{split}
	&e^{-\alpha t}\|X_t\| \xrightarrow[t\to \infty]{\text{a.s.}} \widetilde{H}_\infty;
	\quad \frac{X_t(f_\mathrm s)}{\|X_t\|^{1-\tilde \beta}} \xrightarrow[t\to \infty]{d} \zeta^{f_\mathrm s};
	\\&\frac{X_t(f_\mathrm c)}{\|t X_t\|^{1-\tilde \beta} } \xrightarrow[t\to \infty]{d} \zeta^{f_\mathrm c};
	\quad\frac{X_t(f_\mathrm l) - \mathrm x_t(f_\mathrm l)}{\|X_t\|^{1-\tilde \beta}}
	\xrightarrow[t\to \infty]{d} \zeta^{-f_\mathrm l},
\end{split}
\]
	where $\widetilde H_\infty$ has the distribution of $\{H_{\infty}; \widetilde {\mathbb P}_\mu\}$;
	$\zeta^{f_\mathrm s}$, $\zeta^{f_\mathrm c}$ and $\zeta^{-f_\mathrm l}$ are the $(1+\beta)$-stable random variables described in \eqref{eq:I:R:3}; and
\[
	\mathrm x_t(f)
	:= \sum_{p\in \mathbb Z^d_+:\alpha \tilde \beta>|p|b}\langle f,\phi_p\rangle_\varphi e^{(\alpha-|p|b)t}H^p_{\infty},
	\quad t\geq 0, f\in \mathcal P.
\]
\end{minipage}
\end{equation}
	The above result gives the central limit theorem for $X_t(f)$ if $f\in \mathcal P\setminus\{0\}$ satisfies $\alpha \tilde \beta \leq \kappa_f b$.
	A general  $f \in \mathcal P$ can be decomposed as $f_s + f_c + f_l$ with $f_s \in \mathcal C_\mathrm s$, $f_c \in \mathcal C_\mathrm c$ and $f_l \in \mathcal C_\mathrm l$;
  and if $f\in  \mathcal P$ satisfies $\alpha \tilde \beta > \kappa_f b$, then $f_\mathrm s$ and $f_\mathrm l$ maybe non-zero while $f_\mathrm c \equiv 0$.
	In \cite{RenSongSunZhao2019Stable}, we were not able to establish a central limit theorem in this case. We conjectured there that the limit random variables in \eqref{eq:oldResult} for $f_\mathrm s\in \mathcal C_\mathrm s$, $f_\mathrm c\in \mathcal C_\mathrm c$ and $f_\mathrm l\in \mathcal C_\mathrm l$ are independent.
	Once this asymptotic independence is established, a central limit theorem  for $ X_t(f)$ for all $f\in  \mathcal P$ would follow.

	The main purpose of this note is to show that the limit  random variables in \eqref{eq:oldResult} are independent.

\begin{thm}
\label{thm:M}
	If $\mu\in \mathcal M_\mathrm c(\mathbb R^d)\setminus \{0\}$, $f_\mathrm s\in \mathcal C_\mathrm s$, $f_\mathrm c \in \mathcal C_\mathrm c$ and $f_\mathrm l \in \mathcal C_\mathrm l$, then under $\mathbb {\widetilde P}_\mu$,
\begin{align} \label{eq:M.1}
	&S(t):=
	\Bigg(e^{-\alpha t}\|X_t\|, \frac{X_t(f_\mathrm s)}{\|X_t\|^{1-\tilde \beta}},\frac{X_t(f_\mathrm c)}{\|tX_t\|^{1-\tilde \beta}}, \frac{ X_t(f_\mathrm l) - \mathrm x_t(f_\mathrm l)}{\|X_t\|^{1-\tilde \beta}}
	\Bigg)
	\\&\xrightarrow[t\rightarrow \infty]{d}(\widetilde H_\infty,\zeta^{f_\mathrm s},\zeta^{f_\mathrm c},\zeta^{-f_\mathrm l}),
\end{align}
	where $\mathrm x_t(f_\mathrm l)$ is defined in 
	\eqref{eq:oldResult} 
	with $f$ replaced with $f_\mathrm l$;
	$\widetilde H_\infty$ has the distribution of $\{H_{\infty}; \widetilde {\mathbb P}_\mu\}$;
	$\zeta^{f_\mathrm s}$, $\zeta^{f_\mathrm c}$ and $\zeta^{-f_\mathrm l}$ are the $(1+\beta)$-stable random variables described in \eqref{eq:I:R:3};
	$\widetilde H_\infty$,  $\zeta^{f_\mathrm s}$, $\zeta^{f_\mathrm c}$ and $\zeta^{-f_\mathrm l}$ are independent.
\end{thm}

	As a corollary of this theorem, we get  central limit theorems for $X_t(f)$ for all $f\in \mathcal P$.

\begin{cor}
	Let $\mu\in \mathcal M_c(\mathbb R^d)\setminus \{0\}$ and $f\in \mathcal P$.
	Let  $f=f_\mathrm s + f_\mathrm c + f_\mathrm l$ be the unique decomposition of $f$ with $f_\mathrm s \in \mathcal C_\mathrm s$, $f_\mathrm c \in \mathcal C_\mathrm c$ and $f_\mathrm l \in \mathcal C_\mathrm l$.
	Then under $\widetilde {\mathbb{P}}_{\mu}$, it holds that
\begin{enumerate}
\item
	if $f_\mathrm c\equiv 0$, then
\[
    \frac{ X_t(f) - \mathrm x_t(f)}{\|X_t\|^{1-\tilde \beta}}
    \xrightarrow[t\to \infty]{d} \zeta^{f_\mathrm s}+\zeta^{-f_\mathrm l},
\]
	where $\zeta^{f_\mathrm s}$ and $\zeta^{-f_\mathrm l}$ are the $(1+\beta)$-stable random variables described in \eqref{eq:I:R:3}, $\zeta^{f_\mathrm s}$ and $\zeta^{-f_\mathrm l}$ are independent;
\item
if $f_\mathrm c\not\equiv 0$, then
\[
    \frac{ X_t(f) - \mathrm x_t(f)}{\|tX_t\|^{1-\tilde \beta}}
    \xrightarrow[t\to \infty]{d}\zeta^{f_\mathrm c}.
\]
	where $\zeta^{f_\mathrm c}$ is the $(1+\beta)$-stable random variable described in \eqref{eq:I:R:3}.
\end{enumerate}
\end{cor}

\section{Proof of main result}
\label{proofs of main results}
	We first make some preparations before proving Theorem \ref{thm:M}.
	For every $t\geq 0$ and $f\in \mathcal P$, define
\[
  	Z_t f
  	:= \int_0^t P^\alpha_{t-s}\big( \eta (-i P^\alpha_sf)^{1+\beta}\big)\mathrm ds,
  	\quad \Upsilon^f_t
   	:= \frac{X_{t+1} (f) - X_t(P_1^\alpha f)}{\| X_t\|^{1-\tilde \beta}}.
\]
	Form \cite[Theorem 3.4]{RenSongSunZhao2019Stable} we know that, for each $f\in \mathcal P$, $\langle Z_1f,\varphi\rangle$ is the characteristic exponent of 
the weak limit of $\Upsilon^f_t$.
	For $g = g_\mathrm s + g_\mathrm c + g_\mathrm l\in \mathcal P$ with $g_\mathrm s \in \mathcal C_\mathrm s, g_\mathrm c \in \mathcal C_\mathrm c$ and $g_\mathrm 1 \in \mathcal C_\mathrm l$, we define $\overline {\mathcal P}_g:= \{\theta_\mathrm s T_n g_\mathrm s +  \theta_\mathrm c T_n g_\mathrm c + \theta_\mathrm 1 T_n g_\mathrm 1: n \in \mathbb Z_+, \theta_\mathrm s, \theta_\mathrm c, \theta_\mathrm l \in [-1,1]\}$, where $T_n$ is the operator defined in \eqref{e:Tt}.
	The following Lemma \ref{lem:P:R} can be proved using 
an argument similar to that used in the proof of 
	\cite[Lemma 2.9]{RenSongSunZhao2019Stable}. We omit the details here.
\begin{lem}
\label{lem:P:R}
	For any $g\in \mathcal P$ there exists $h \in \mathcal P^+$ such that for all $ f \in \overline {\mathcal P}_g$ and $t\geq 0$, we have $ | P_t (Z_1 f - \langle Z_1 f, \varphi \rangle )| \leq e^{-bt} h$.
\end{lem}

	The following result is a generalization of \cite[Proposition 3.5]{RenSongSunZhao2019Stable}, whose proof is similar to that of \cite[Proposition 3.5]{RenSongSunZhao2019Stable},
with Lemma \ref{lem:P:R} replacing the role of \cite[Lemma 2.9]{RenSongSunZhao2019Stable}.
	Let $(\mathscr F_t)_{t\geq 0}$ be the natural filtration of $X$.
	\begin{prop}
\label{thm:Key}
	For any $\mu \in \mathcal M_c(\mathbb R^d)$  and $g \in \mathcal P$, there exist $C,\delta>0$ such that for all $t\geq 1$ and $f \in \overline{\mathcal P}_g$, we have
\[
	\mathbb P_\mu
	\Big[ \big|\mathbb P_\mu [e^{i\Upsilon^f_t} - e^{\langle Z_1f, \varphi\rangle}; D^c | \mathscr F_t ]  \big|\Big]
	\leq C e^{- \delta t}.
\]
\end{prop}

	The following generalization of \cite[Proposition 3.5]{RenSongSunZhao2019Stable} will be used later in the proof of  Theorem \ref{thm: II}, a special case of Theorem \ref{thm:M}.
	Note that the constants $C$ and $\delta$ in the next result depend only on $f,g\in \mathcal P$ and $\mu\in \mathcal M_c(\mathbb R^d)$, do not depend on $n_1, n_2, f_j, g_j$ and $t$ (as long as $t\geq n_1+1$).

\begin{prop} \label{cor:MI}
	For any  $f,g\in \mathcal P$ and $\mu\in \mathcal M_c(\mathbb R^d)$, there exist $C,\delta>0$ such that for all $n_1,n_2 \in \mathbb Z_+$, $(f_j)_{j=0}^{n_1}\subset \overline{\mathcal P}_f$, $(g_j)_{j=0}^{n_2}\subset \overline{\mathcal P}_g$ and $t\geq n_1+1$, we have
\begin{equation}
\label{32corollary}
	\Big|\mathbb{\widetilde{P}}_{\mu}\Big[  \Big(\prod_{k=0}^{n_1}e^{i \Upsilon^{f_k}_{t-k-1}}\Big)  \Big( \prod_{k=0}^{n_2}e^{i \Upsilon^{g_k}_{t+k} } \Big) \Big]  -  \Big(\prod_{k=0}^{n_1} e^{\langle Z_1f_k, \varphi\rangle}\Big) \Big(\prod_{k=0}^{n_2} e^{\langle Z_1g_k, \varphi\rangle}\Big) \Big|
	\leq C e^{-\delta (t-n_1)}.
\end{equation}
\end{prop}
\begin{proof}
	In this proof, we fix $f,g\in \mathcal P$, $\mu\in \mathcal M_c(\mathbb R^d)$, $n_1,n_2 \in \mathbb Z_+$, $(f_j)_{j=0}^{n_1}\subset \overline{\mathcal P}_f$, $(g_j)_{j=0}^{n_2}\subset \overline{\mathcal P}_g$ and $t\geq n_1 + 1$.
 	For any $k_1 \in \{-1,0,\dots,n_1\}$ and $k_2 \in \{-1,0,\dots,n_2\}$,  define
\[
    a_{k_1,k_2}
    :=  \mathbb{\widetilde{P}}_{\mu}\Big[ \Big(\prod_{j=k_1+1}^{n_1} e^{i\Upsilon_{t-j-1}^{f_j}} \Big)  \Big(\prod_{j=0}^{k_2}e^{i\Upsilon_{t+j}^{g_j}}\Big) \Big] \Big(\prod_{j=0}^{k_1}e^{\langle Z_1 f_j, \varphi\rangle}\Big) \Big(\prod_{j=k_2+1}^{n_2} e^{ \langle Z_1g_j,\varphi \rangle} \Big),
\]
	where we used the convention that $\prod_{j=0}^{-1} =1.$
  	Then for all  $k_2 \in \{0,\dots,n_2\}$, we have
\begin{equation}
\label{eq:MI.1}
\begin{multlined}
	 a_{-1,k_2} - a_{-1,k_2-1}
	 = \frac{1}{\mathbb{P}_{\mu}(D^c)}  \Big(\prod_{j=k_2+1}^{n_2}e^{\langle Z_1g_j, \varphi\rangle}\Big) \times {}
	\\ \mathbb{P}_{\mu}\Big[\Big(\prod_{j=0}^{n_1}e^{i\Upsilon_{t-j-1}^{f_j}}\Big)\Big(\prod_{j=0}^{k_2-1} e^{i\Upsilon_{t+j}^{g_j}}\Big) \mathbb P_\mu[e^{i\Upsilon^{g_{k_2}}_{t+k_2}}-e^{\langle Z_1g_{k_2}, \varphi\rangle}; D^c|\mathscr F_{t+k_2}] \Big].
\end{multlined}
\end{equation}
	By Proposition \ref{thm:Key}, there exist $C_0,\delta_0 >0$,  depending only on $\mu$ and $g$, such that  for each $k_2 \in \{0, \dots, n_2 \}$,
\begin{align}
    &| a_{-1,k_2} - a_{-1,k_2-1}|
    \overset{\eqref{eq:MI.1}}\leq \mathbb{P}_{\mu}(D^c)^{-1}\mathbb{P}_{\mu}\Big[\big|\mathbb P_\mu[e^{i\Upsilon^{g_{k_2}}_{t+k_2}}-e^{\langle Z_1g_{k_2}, \varphi\rangle}; D^c|\mathscr F_{t+k_2}]\big|\Big]
    \\&\label{eq:MI.15} \leq C_0 e^{-\delta_0 (t+k_2)}.
\end{align}
	Similarly, for any $k_1 \in \{0, \dots , n_1\}$,
\begin{equation}
\label{eq:MI.2}
\begin{multlined}
	a_{k_1-1,-1} - a_{k_1,-1}
	=  \frac{1}{\mathbb P_\mu(D^c)} \Big(\prod_{j=0}^{k_1-1}e^{\langle Z_1 f_j, \varphi\rangle}\Big) \Big(\prod_{j=0}^{n_2} e^{ \langle Z_1g_j,\varphi \rangle} \Big) \times {}
	\\ \qquad \mathbb{P}_{\mu}\Big[ \mathbb P_\mu\big[e^{i\Upsilon_{t-k_1-1}^{f_{k_1}}} -e^{\langle Z_1 f_{k_1}, \varphi\rangle}  ; D^c \big| \mathscr F_{t-k_1 - 1}\big] \prod_{j=k_1+1}^{n_1} e^{i\Upsilon_{t-j-1}^{f_j}} \Big] .
\end{multlined}
\end{equation}
	By Proposition \ref{thm:Key}, there exist $C_1,\delta_1 >0$, depending only on $\mu$ and $f$, such that for any $k_1 \in \{0,\dots,n_1\}$,
\begin{align}
    &| a_{k_1-1,-1} - a_{k_1,-1}|
    \overset{\eqref{eq:MI.2}}\leq \frac{1}{\mathbb{P}_{\mu}(D^c)}\mathbb{P}_{\mu}\Big[\big|\mathbb P_\mu[e^{i\Upsilon^{f_{k_1}}_{t-k_1-1}}-e^{\langle Z_1f_{k_1}, \varphi\rangle}; D^c|\mathscr F_{t-k_1-1}]\big|\Big]
    \\&\label{eq:MI.25} \leq C_1 e^{-\delta_1 (t-k_1)}.
\end{align}
	Therefore, there exist $C,\delta >0$, depending only on $f,g$ and $\mu$, such that
\begin{align}
    &\text{LHS of \eqref{32corollary}}
    = \left|a_{-1,n_2}-a_{n_1,-1}\right|
    \leq \sum_{k=0}^{n_1}\left|a_{k-1,-1}-a_{k,-1}\right|+\sum_{k=0}^{n_2}\left|a_{-1,k}-a_{-1,k-1}\right|
    \\& \overset{\eqref{eq:MI.15},\eqref{eq:MI.25}}\leq \sum_{k=0}^{n_1} C_1 e^{-\delta_1 (t-k)}+\sum_{k=0}^{n_2} C_0 e^{-\delta_0 (t+k)}
  	\leq C e^{- \delta (t-n_1)}.
    \qedhere
\end{align}
\end{proof}

	The following analytic result is elementary, and will also be used in the proof of Theorem \ref{thm: II}.
	\begin{lem}
\label{ineq: analysis}
	There exists a constant $C>0$, such that for any $x,y \in \mathbb R$,
\[
    |(x+y)^{1+\beta}-x^{1+\beta}-y^{1+\beta}|
    \leq C(|x||y|^{\beta}+|x|^{\beta}|y|).
\]
\end{lem}

	In the remainder of this section, we  fix $\mu \in \mathcal M_c(\mathbb R^d)\setminus\{0\}$, $f_\mathrm s\in \mathcal C_\mathrm s$,  $f_\mathrm c\in \mathcal C_\mathrm c$ and $f_\mathrm l\in \mathcal C_\mathrm l$.
	For every $t\geq 1$, define
\begin{align}
	R(t)
	&:=\Big( \frac{X_t(f_\mathrm s)}{\|X_t\|^{1-\tilde \beta}},\frac{X_t(f_\mathrm c)}{\|tX_t\|^{1-\tilde \beta}},\frac{ X_t(f_\mathrm l) - \mathrm x_t(f_\mathrm l)}{\|X_t\|^{1-\tilde \beta}}\Big),
        \\R'(t)
    &=\big(R'_\mathrm s(t), R'_\mathrm c(t), R'_\mathrm l(t)\big)
        \\& :=\Big(\sum_{k=0}^{\lfloor t-\ln t \rfloor} \Upsilon_{t-k-1}^{T_k \tilde f_\mathrm s},t^{\tilde \beta - 1}\sum_{k=0}^{\lfloor t-\ln t \rfloor} \Upsilon_{t-k-1}^{T_{k} \tilde f_\mathrm c},\sum_{k = 0}^{\lfloor t^2 \rfloor} \Upsilon_{t+k}^{- T_k \tilde f_\mathrm l}\Big),
\end{align}
where $T_k$ is the operator defined in 
\eqref{e:Tt}, $\mathrm x_t(f_\mathrm l)$ is defined in \eqref{eq:oldResult} 
with $f$ replaced with $f_\mathrm l$, $\tilde f_\mathrm s:=e^{\alpha(\tilde \beta - 1)} f_\mathrm s$, $\tilde f_\mathrm c:=e^{\alpha(\tilde \beta - 1)} f_\mathrm c$ and $\tilde f_\mathrm l := \sum_{p\in \mathbb Z^d_+: \alpha\tilde\beta>|p|b}e^{-(\alpha - |p|b)}\langle f_\mathrm l, \phi_p \rangle_\varphi \phi_p$.
	The following result is a special case of Theorem \ref{thm:M}.
\begin{thm}\label{thm: II}
	Under $\mathbb{\widetilde{P}}_{\mu}$, $R(t) \xrightarrow[t\rightarrow\infty]{d}(\zeta^{f_\mathrm s},\zeta^{f_\mathrm c},\zeta^{-f_\mathrm l})$, where $\zeta^{f_\mathrm s},\zeta^{f_\mathrm c}$ and $\zeta^{-f_\mathrm l}$ are the $(1+\beta)$-stable random variables described in \eqref{eq:I:R:3}, and $\zeta^{f_\mathrm s},\zeta^{f_\mathrm c}$ and $\zeta^{-f_\mathrm l}$ are independent.
\end{thm}
\begin{proof}
	In this proof, we always work under $\mathbb{\widetilde{P}}_{\mu}$.
	According to the proof of 
\cite[Theorem 1.6]{RenSongSunZhao2019Stable} and the fact that the convergence in
 probability of random vectors to the zero vector is equivalent to the convergence of each components of the random vectors to zero, we have
\[
	R(t)-R'(t)
	\xrightarrow[t\to \infty]{\text{in probability}} 0.
\]
	With the help of Slutsky's theorem, what is left to show is that,
\begin{equation} \label{lem:UOT}
	R'(t)
	\xrightarrow[t\to \infty]{d}(\zeta^{f_\mathrm s},\zeta^{f_\mathrm c},\zeta^{-f_\mathrm l}).
\end{equation}

	Now we prove \eqref{lem:UOT}.
	Since $\Upsilon_t^f$ is linear in $f$, for each $t\geq 1$,
\[
        \widetilde{\mathbb P}_{\mu}\Big[\exp\Big(i \sum_{j=\mathrm s,\mathrm c,\mathrm l} R'_j(t)\Big)\Big]
	= \widetilde{\mathbb P}_{\mu}\Big[\exp\Big(i\sum_{k=0}^{\lfloor t-\ln t \rfloor}\Upsilon_{t-k-1}^{T_k(\tilde{f_\mathrm s}+t^{\tilde{\beta}-1} \tilde{f}_{\mathrm c})}\Big)\exp\Big(i\sum_{k=0}^{\lfloor t^2 \rfloor}\Upsilon_{t+k}^{-T_k\tilde{f}_{\mathrm l}}\Big)\Big].
\]
	Note that $\{T_k(\tilde f_\mathrm s + t^{\tilde\beta - 1}\tilde f_{\mathrm c}):k\in \mathbb Z_+, t\geq 1\}\subset \overline{\mathcal P}_{\tilde f_\mathrm s + \tilde f_{\mathrm c}}$ and $\{-T_k\tilde{f}_{\mathrm l}: k\in \mathbb Z_+\}\subset \overline{\mathcal P}_{\tilde f_\mathrm l}$.
	Therefore, we can use Proposition \ref{cor:MI} with $f$ 
taken as $\tilde{f_\mathrm s}+ \tilde{f}_{\mathrm c}$ and $g$ taken as
	$\tilde{f}_{\mathrm l}$
	to get that there exist $C_1,\delta_1 > 0$ such that for every $t>e$ (which implies $t\geq \lfloor t - \ln t\rfloor +1$),
\begin{align}
       &\begin{multlined}
  	\bigg|\widetilde{\mathbb P}_{\mu}\Big[\exp\Big(i \sum_{j=\mathrm s,\mathrm c,\mathrm l}R'_j(t)\Big)\Big] - {}
  	\\ \exp\Big(\sum_{k=0}^{\lfloor t-\ln t \rfloor} \big\langle Z_1\big(T_{k}(\tilde f_\mathrm s+t^{\tilde{\beta}-1}\tilde{f}_{\mathrm c})\big), \varphi \big\rangle \Big)\exp\Big(\sum_{k=0}^{\lfloor t^2 \rfloor}\langle Z_1(-T_k\tilde{f}_{\mathrm l}),\varphi\rangle\Big)\bigg|
  	\end{multlined}
      \\ &\leq C_1 e^{-\delta_1(t - \lfloor t - \ln t\rfloor)}.
\end{align}
	We claim that
\begin{align}
\label{eq:UOT.1}
	&\lim_{t\rightarrow\infty}\exp\Big(\sum_{k=0}^{\lfloor t-\ln t \rfloor}\big \langle Z_1\big(T_{k}(\tilde f_\mathrm s+t^{\tilde{\beta}-1}\tilde{f}_\mathrm c)\big), \varphi\big\rangle \Big)\exp\Big(\sum_{k=0}^{\lfloor t^2 \rfloor}\langle Z_1(-T_k\tilde{f}_\mathrm l),\varphi\rangle\Big)
	\\& =\exp(m[f_\mathrm s]+m[f_\mathrm c]+m[-f_\mathrm l]).
\end{align}
	Given this claim, we have
\[
	\widetilde{\mathbb P}_{\mu}\Big[\exp\Big(i \sum_{j=\mathrm s,\mathrm c,\mathrm l} R'_{j}(t)\Big)\Big]
	\xrightarrow[t\to \infty]{} \exp(m[f_\mathrm s]+m[f_\mathrm c]+m[-f_\mathrm l]).
\]
	Since $R'_{j}(t)$ are linear 
	 in $f_j \in \mathcal C_j~(j=\mathrm s,\mathrm c,\mathrm l)$, replacing $f_j$ with $\theta_j f_j$, we immediately get \eqref{lem:UOT}.

	Now we prove the claim \eqref{eq:UOT.1}.
	For every $f\in \mathcal C_\mathrm s  \oplus \mathcal C_\mathrm c$ and $n\in \mathbb Z_+$,
 \begin{align}
  	& \sum_{k=0}^n \langle Z_1 T_{k} \tilde f, \varphi \rangle
  	= \sum_{k=0}^n \int_0^1 \big\langle P_u^\alpha \big(\eta(-iP_{1 - u}^\alpha T_k \tilde f)^{1+\beta}\big), \varphi\big\rangle \mathrm du
 	\\& = \sum_{k=0}^n \int_0^1 e^{\alpha u} \langle \eta (-iP_{1 - u}^\alpha T_{k}\tilde f)^{1+\beta}, \varphi \rangle \mathrm du
 	= \sum_{k=0}^n \int_0^1 \langle \eta (-iT_{k+1 - u} f)^{1+\beta}, \varphi\rangle \mathrm du
 	\\&= \int_0^{n+1} \langle  \eta (-iT_{u} f)^{1+\beta}, \varphi\rangle \mathrm du
 	= m_{n+1}[f],
\end{align}
	where $\tilde f=e^{\alpha(\tilde \beta - 1)} f$.
	Therefore, for any $t\geq 1$,
\begin{equation} \label{eq:UOT.15}
	\sum_{k=0}^{\lfloor t-\ln t \rfloor} \langle Z_1T_{k}(\tilde f_\mathrm s+t^{\tilde{\beta}-1}\tilde{f}_\mathrm c), \varphi\rangle
	= \eta \int_0^{\lfloor t-\ln t \rfloor+1}\big\langle \big(-iT_u(f_\mathrm s+t^{\tilde{\beta}-1}f_\mathrm c)\big)^{1+\beta},\varphi \big\rangle \mathrm du.
\end{equation}
	Note that for each $u\geq 0$, $T_uf_\mathrm c=f_\mathrm c$.
	Also note that according to
 	Step 1 in the proof of \cite[Lemma 2.6]{RenSongSunZhao2019Stable}, there exist $\delta> 0$ and $h\in \mathcal P$ (depending only on $f_\mathrm s$) such that for each $u\geq 0$, $|T_u f_\mathrm s|\leq e^{-\delta u}h$.
	It follows from Lemma \ref{ineq: analysis} that there exists $C>0$ such that for all $u\geq 0$ and $t\geq 0$,
\begin{align}
  	&|(-i(T_uf_\mathrm s+t^{\tilde{\beta}-1}T_uf_\mathrm c))^{1+\beta}-(-iT_uf_\mathrm s)^{1+\beta}-(-it^{\tilde{\beta}-1}T_uf_\mathrm c)^{1+\beta}|
  	\\&  = |-i|^{1+\beta} |(T_u f + t^{\tilde \beta -1} T_u f_\mathrm c)^{1+\beta} - (T_u f_\mathrm s)^{1+\beta} - (t^{\tilde \beta - 1}T_u f_\mathrm c)^{1+\beta}|
  	\\&\overset{\text{Lemma \ref{ineq: analysis}}}
  	\leq  C(t^{-\frac{\beta}{1+\beta}}|T_uf_\mathrm s||T_uf_\mathrm c|^{\beta}+t^{-\frac{1}{1+\beta}}|T_uf_\mathrm s|^{\beta}|T_uf_\mathrm c|)
  	\\&\label{eq:UOT.21}\leq C(t^{-\frac{\beta}{1+\beta}}e^{-\delta u}h|f_\mathrm c|^{\beta}+t^{-\frac{1}{1+\beta}}e^{-\delta\beta u}h^{\beta}|f_\mathrm c|).
\end{align}
	This means that there exists $C_1 >0$ such that for all $t\geq 1$,
\begin{align}
	&\Big|\Big(\sum_{k=0}^{\lfloor t-\ln t \rfloor} \langle Z_1T_{k}(\tilde f_\mathrm s+t^{\tilde{\beta}-1}\tilde{f}_c), \varphi\rangle \Big)-m_{\lfloor t-\ln t \rfloor+1}[f_\mathrm s]-\frac{1}{t}m_{\lfloor t-\ln t \rfloor+1}[f_\mathrm c]\Big|
	\\&
\begin{multlined}
	\overset{\eqref{eq:UOT.15},\eqref{eq:M.13}}\leq\Big| \eta \int_0^{\lfloor t-\ln t \rfloor+1}\big\langle \big(-iT_u(f_\mathrm s+t^{\tilde{\beta}-1}f_\mathrm c)\big)^{1+\beta},\varphi \big\rangle \mathrm du - {}
	\\ \eta \int_0^{\lfloor t-\ln t \rfloor+1} \langle (-iT_u f_\mathrm s)^{1+\beta}, \varphi \rangle  \mathrm du - \eta \int_0^{\lfloor t-\ln t \rfloor+1} \langle (-iT_u f_\mathrm c)^{1+\beta}, \varphi \rangle  \mathrm du\Big|
\end{multlined}
	\\ & \overset{\eqref{eq:UOT.21}}\leq C_1\int_0^{\lfloor t-\ln t \rfloor+1}\langle t^{-\frac{\beta}{1+\beta}}e^{-\delta u}h|f_\mathrm c|^{\beta}+t^{-\frac{1}{1+\beta}}e^{-\delta\beta u}h^{\beta}|f_\mathrm c|, \varphi \rangle \mathrm du
	\\& \leq C_1 t^{-\frac{\beta}{1+\beta}} \langle h|f_\mathrm c|^{\beta},\varphi \rangle \int_0^\infty e^{-\delta u} \mathrm du + C_1 t^{-\frac{1}{1+\beta}} \langle h^\beta|f_\mathrm c|,\varphi \rangle \int_0^\infty e^{-\delta \beta u} \mathrm du
	\\& \xrightarrow[t\rightarrow \infty]{} 0.
\end{align}
	Combining this with \eqref{eq:I:R:3}, we get that
\begin{equation}
\label{eq:UOT.4}
	\lim_{t\rightarrow \infty}\exp\Big(\sum_{k=0}^{\lfloor t-\ln t \rfloor} \langle Z_1T_{k}(\tilde f_\mathrm s+t^{\tilde{\beta}-1}\tilde{f}_c), \varphi\rangle \Big)  = \exp( m[f_\mathrm s]+m[f_\mathrm c]).
\end{equation}
	Also note that according to the Step 1 in the Proof of Theorem 1.6.(3) in \cite{RenSongSunZhao2019Stable}, we have
\begin{equation}
\label{eq:UOT.5}
	\lim_{t\rightarrow \infty}\exp\Big(\sum_{k=0}^{\lfloor t^2 \rfloor}\langle Z_1(-T_k\tilde{f}_\mathrm l),\varphi\rangle\Big) =\exp(m[-f_\mathrm l]).
\end{equation}
	Thus the desired claim follows from \eqref{eq:UOT.4} and \eqref{eq:UOT.5}.
\end{proof}

\begin{proof}[Proof of Theorem \ref{thm:M}]
 	We first recall some facts about weak convergence which will be used later.
	For any bounded Lipschitz function $f:\mathbb R^d\mapsto \mathbb R$, let
\[
	\|f\|_L
	:=\sup_{x\neq y}\frac{|f(x)-f(y)|}{|x-y|}
\]
	and $\|f\|_{BL}:= \|f\|_{\infty}+\|f\|_L.$
	For any probability distributions $\mu_1$ and $\mu_2$ on $\mathbb R^d$, define
\[
	d(\mu_1,\mu_2)
  	:=\sup\Big\{\Big|\int f \mathrm d\mu_1-\int f \mathrm d\mu_2\Big|:\|f\|_{BL}\leq 1\Big\}.
\]
	Then $d$ is a metric. It follows from \cite[Theorem 11.3.3]{Dudley2002} that the topology generated by $d$ is equivalent to the weak convergence topology.
	Using the definition, we can easily see that, if $\mu_1$ and $\mu_2$ are the distributions of two $\mathbb R^d $-valued random variables $X$ and $Y$ respectively, defined on same probability space, then
\begin{align}
\label{ineq: distribution control}
  	d(\mu_1,\mu_2)
  	\leq \mathbb E|X-Y|.
\end{align}
	In this proof, let us fix $\mu\in \mathcal M_\mathrm c(\mathbb R^d)\setminus \{0\}$, $f_\mathrm s\in \mathcal C_\mathrm s$, $f_\mathrm c \in \mathcal C_\mathrm c$ and $f_\mathrm l \in \mathcal C_\mathrm l$.
	Recall that $S(t)~(t\geq 0)$ is given by \eqref{eq:M.1}.
	For every $r,t> 0$, let
\begin{align}
	S(t,r)
	:=\Big(e^{-\alpha t}\|X_t\|,& \frac{X_{t+r}(f_\mathrm s)}{\|X_{t+r}\|^{1-\tilde{\beta}}}, \frac{X_{t+r}(f_\mathrm c)}{\|(t+r)X_{t+r}\|^{1-\tilde{\beta}}}, \frac{X_{t+r}(f_\mathrm l)-\mathrm x_{t+r}(f_\mathrm l) }{\|X_{t+r}\|^{1-\tilde{\beta}}}\Big),
\end{align}
	and
\begin{align}
	\widetilde{S}(t,r)
	= (e^{-\alpha (t+r)}\|X_{t+r}\|-e^{-\alpha t}\|X_t\|,0,0,0),
\end{align}
	where, for any $t>0$, $\mathrm x_t(f_\mathrm l)$ is defined in \eqref{eq:oldResult} with $f$ replaced with $f_\mathrm l$.
	Then $S(t+r)=S(t,r)+\widetilde{S}(t,r)$.
	We claim that
\begin{equation}
\label{eq:M.3}
\begin{minipage}{0.9\textwidth}
	for each $t> 0$, under $\widetilde{\mathbb P}_{\mu}$, we have
\[
	S(t,r)
	\xrightarrow[r\rightarrow \infty]{d} (\widetilde H_t,\zeta^{f_\mathrm s},\zeta^{f_\mathrm c},\zeta^{-f_\mathrm l}),
\]
	where $\widetilde H_t$ has the distribution of $\{e^{-\alpha t} \|X_t\|; \widetilde {\mathbb P}_\mu\}$, $\zeta^{f_\mathrm s},\zeta^{f_\mathrm c}$ and $\zeta^{-f_\mathrm l}$ are the $(1+\beta)$-stable random variables described in \eqref{eq:I:R:3}, 	and $\widetilde H_t,\zeta^{f_\mathrm s},\zeta^{f_\mathrm c}$ and $\zeta^{-f_\mathrm l}$ are independent.
\end{minipage}
\end{equation}

	For every $r,t\geq 0$, let $\mathcal D(r)$ and $\mathcal D(r,t)$ be the distributions of $S(r)$ and $S(t,r)$ under $\widetilde{\mathbb P}_{\mu}$ respectively;
	let $\widetilde{\mathcal D}(t)$ and $\mathcal D$ be the distributions of $(\widetilde H_t,\zeta^{f_\mathrm s},\zeta^{f_\mathrm c},\zeta^{-f_\mathrm l})$ and $(\widetilde H_\infty,\zeta^{f_\mathrm s},\zeta^{f_\mathrm c},\zeta^{-f_\mathrm l})$, respectively.
	Then for each $\gamma\in (0,\beta)$, there exist constant $C>0$ such that for every $t > 0$,
\begin{align}
 	&\varlimsup_{r\rightarrow \infty}d(\mathcal D(t+r),\mathcal D)
 	\\& \overset{\rm triangle~inequality}\leq \varlimsup_{r\rightarrow \infty}\Big( d\big(\mathcal D(t+r),\mathcal D(t,r)\big)+d\big(\mathcal D(t,r),\widetilde{\mathcal D}(t)\big)+d\big(\widetilde{\mathcal D}(t),\mathcal D\big)\Big)
 	\\&\overset{\eqref{ineq: distribution control}} \leq \varlimsup_{r\rightarrow \infty}\widetilde{\mathbb P}_{\mu}[|S(t+r) - S(t,r)|]+ \varlimsup_{r\rightarrow \infty} d\big(\mathcal D(t,r),\widetilde{\mathcal D}(t)\big) +\widetilde{\mathbb P}_{\mu}[|H_t-H_{\infty}|]
	\\&\overset{\eqref{eq:M.3}}\leq \varlimsup_{r\rightarrow \infty}\widetilde{\mathbb P}_{\mu}[|H_t - H_{t+r}|]+\widetilde{\mathbb P}_{\mu}[|H_t-H_{\infty}|]
  	\\&\overset{\text{H\"older inequality}}\leq \varlimsup_{r\rightarrow \infty}\mathbb P_{\mu}(D^c)^{-1}(\|H_t-H_{t+r}\|_{L_{1+\gamma} (\mathbb P_\mu)}+\|H_t-H_{\infty}\|_{L_{1+\gamma} (\mathbb P_\mu)})
	\\&\label{eq:M.4}\overset{\text{\cite[Lemma 3.3]{RenSongSunZhao2019Stable}}}\leq C e^{-\alpha \tilde \gamma t}.
\end{align}
	Therefore,
\begin{align}
 	&\varlimsup_{r\rightarrow \infty}d\big(\mathcal D(r),\mathcal D\big)
 	=\varlimsup_{t\to \infty}\varlimsup_{r\rightarrow \infty}d\big(\mathcal D(t+r),\mathcal D\big)
	\overset{\eqref{eq:M.4}}\leq \varlimsup_{t\rightarrow \infty}Ce^{-\alpha \tilde \gamma t} = 0.
\end{align}
	The desired result now follows immediately.

	Now we prove the claim \eqref{eq:M.3}.
	For every $r,t >0$, let
\[
	\theta,\theta_\mathrm s,\theta_\mathrm c,\theta_\mathrm l \in \mathbb R
	\mapsto k(\theta,\theta_\mathrm s,\theta_\mathrm c,\theta_\mathrm l,r,t)
\]
	be the characteristic function of  $S(t,r)$ under $\widetilde{\mathbb P}_{\mu}$.
	Then for each $\theta,\theta_\mathrm s,\theta_\mathrm c,\theta_\mathrm l\in \mathbb R$ and $r,t> 0$,
\begin{align}
	&k(\theta,\theta_\mathrm s,\theta_\mathrm c,\theta_\mathrm l,r,t)
	=\widetilde{\mathbb P}_{\mu}\big[\exp\big( i\theta e^{-\alpha t}\|X_t\|+A(\theta_\mathrm s,\theta_\mathrm c,\theta_\mathrm l,r,t,\infty)\big)\big]\\
	&\label{eq:M.5}\overset{\text{bounded convergence}}=\lim_{u\rightarrow \infty}\frac{1}{\mathbb P_{\mu}(D^c)}\mathbb P_{\mu}\big[\exp\big( i\theta e^{-\alpha t}\|X_t\|+A(\theta_\mathrm s,\theta_\mathrm c,\theta_\mathrm l,r,t,u)\big);D^c\big],
\end{align}
	where for each $u\in [0,\infty]$,
\begin{align}
	&A(\theta_\mathrm s,\theta_\mathrm c,\theta_\mathrm l,r,t,u)
	\\&:=i\theta_\mathrm s \frac{X_{t+r}(f_\mathrm s)}{\|X_{t+r}\|^{1-\tilde{\beta}}} + i\theta_\mathrm c \frac{X_{t+r}(f_\mathrm c)}{\|(t+r)X_{t+r}\|^{1-\tilde{\beta}}} + i\theta_\mathrm l  \frac{ X_{t+r}(f_\mathrm l)- \mathbb P_\mu[\mathrm x_{t+r}(f_\mathrm l)|\mathscr F_u]}{\|X_{t+r}\|^{1-\tilde{\beta}}}
	\\&\label{eq:M.6}
\begin{multlined}
	=i\theta_\mathrm s \frac{X_{t+r}(f_\mathrm s)}{\|X_{t+r}\|^{1-\tilde{\beta}}} + \frac{i\theta_\mathrm c}{(t+r)^{1-\tilde{\beta}}} \frac{X_{t+r}(f_\mathrm c)}{\|X_{t+r}\|^{1-\tilde{\beta}}} + {}
	\\\quad  i\theta_\mathrm l  \frac{ X_{t+r}(f_\mathrm l) - \sum_{p\in \mathbb Z_+^d:\alpha\tilde{\beta}>|p|b} e^{(\alpha - |p|b)(t+r)} e^{-(\alpha - |p|b)u}X_u(\phi_p) }{\|X_{t+r}\|^{1-\tilde{\beta}}}.
\end{multlined}
\end{align}
	Now for each $t>0$, we get
\begin{align}
	& \lim_{r\to \infty}k(\theta,\theta_\mathrm s,\theta_\mathrm c,\theta_\mathrm l,r,t)
	\\&\overset{\eqref{eq:M.5}}=\lim_{r\to \infty}\lim_{u\rightarrow \infty}\frac{1}{\mathbb P_{\mu}(D^c)}\mathbb P_{\mu}\big[\exp\{i\theta e^{-\alpha t}\|X_t\|\} \mathbf 1_{\|X_t\|>0} \mathbb P_{\mu}[\exp\{A(\theta_\mathrm s,\theta_\mathrm c,\theta_\mathrm l,r,t,u)\}\mathbf 1_{D^c}|\mathscr F_t]\big]
	\\&
\begin{multlined}
	\overset{\text{\eqref{eq:M.6}, Markov property}}=\lim_{r\to \infty}\lim_{u\rightarrow\infty}\frac{1}{\mathbb P_{\mu}(D^c)} \mathbb P_{\mu}\Bigg[\exp\{ i\theta e^{-\alpha t}\|X_t\|\}\mathbf 1_{\|X_t\|>0} \times {}
	\\ \mathbb P_{X_t}\bigg[\exp\bigg\{A\bigg(\theta_\mathrm s,\theta_\mathrm c\Big(\frac{r}{t+r}\Big)^{1-\tilde \beta},\theta_\mathrm l,r,0,u-t\bigg)\bigg\}\mathbf 1_{D^c}\bigg]\Bigg]
\end{multlined}
	\\&\begin{multlined}\overset{\text{bounded convergence}}=\lim_{r\to \infty} \mathbb P_{\mu}\Bigg[\exp\{ i\theta e^{-\alpha t}\|X_t\|\} \mathbf 1_{\|X_t\|>0}\frac{\mathbb P_{X_t}(D^c)}{\mathbb P_{\mu}(D^c)} \times {}
	\\\widetilde{\mathbb P}_{X_t}\bigg[\exp\bigg\{A\bigg(\theta_\mathrm s,\theta_\mathrm c\Big(\frac{r}{t+r}\Big)^{1-\tilde \beta},\theta_\mathrm l,r,0,\infty\bigg)\bigg\}\bigg]\Bigg].
\end{multlined}
	\\& \overset{\text{Theorem \ref{thm: II}}}=  \mathbb P_{\mu}\Big[\exp\{ i\theta e^{-\alpha t}\|X_t\|\} \mathbf 1_{\|X_t\|>0}\frac{\mathbb P_{X_t}(D^c)}{\mathbb P_{\mu}(D^c)} \Big]\Big(\prod_{j=\mathrm s,\mathrm c}\exp\{m[\theta_j f_j]\}\Big)\exp\{m[-\theta_\mathrm l f_\mathrm l]\}
	\\&=\widetilde{\mathbb P}_{\mu}[\exp\{i\theta e^{-\alpha t}\|X_t\|\}]
	\Big(\prod_{j=\mathrm s,\mathrm c}\exp\{m[\theta_j f_j]\}\Big)\exp\{m[-\theta_\mathrm l f_\mathrm l]\}.
	\qedhere
\end{align}
\end{proof}

\end{document}